\documentclass[a4paper,12pt]{article}
\usepackage[top=2.5cm,bottom=2.5cm,left=2.5cm,right=2.5cm]{geometry}
\usepackage{cite,amsmath,amssymb}
    \usepackage[margin=1cm,%
                font=small,%
                format=hang,%
                labelsep=period,%
                labelfont=bf]{caption}
\usepackage[algoruled,linesnumbered]{algorithm2e}
\usepackage{algorithmic}

\usepackage{amsfonts,epsf,here,graphicx}
\usepackage{latexsym,multirow,rotating}
\usepackage{pgf,tikz,subfigure}
\usepackage{epstopdf}

\newtheorem{theorem}{\bf Theorem}[section]

\newtheorem{lemma}[theorem]{\bf Lemma}

\newtheorem{remark}[theorem]{\bf Remark}

\newcommand{\qed}{\hfill $\square$ \bigskip}

\begin{document}

\baselineskip=0.30in
\vspace*{40mm}

\begin{center}
{\LARGE \bf The Wiener polarity index of benzenoid systems and nanotubes}
\bigskip \bigskip

{\large \bf Niko Tratnik
}
\bigskip\bigskip

\baselineskip=0.20in

\textit{Faculty of Natural Sciences and Mathematics, University of Maribor, Slovenia} \\
{\tt niko.tratnik@um.si, niko.tratnik@gmail.com}
\medskip

\bigskip\medskip

(Received November 9, 2017)

\end{center}

\noindent
\begin{center} {\bf Abstract} \end{center}

\vspace{3mm}\noindent
In this paper, we consider a molecular descriptor called the Wiener polarity index, which is defined as the number of unordered pairs of vertices at distance three in a graph. Molecular descriptors play a fundamental role in chemistry, materials engineering, and in drug design since they can be correlated with a large number of physico-chemical properties of molecules.  As the main result, we develop a method for computing the Wiener polarity index for two basic and most commonly studied families of molecular graphs, benzenoid systems and carbon nanotubes. The obtained method is then used  to find a closed formula for the Wiener polarity index of any benzenoid system. Moreover, we also compute this index for zig-zag and armchair nanotubes.


\baselineskip=0.30in


 \medskip\noindent
 {\bf AMS Subj. Class:} 92E10, 05C12, 05C90

\section{Introduction}

Benzenoid systems (also called hexagonal systems) represent a mathematical model for molecules called benzenoid hydrocarbons and form one of the most important class of chemical graphs \cite{gucy-89}. Similarly, carbon nanotubes are carbon compounds with a cylindrical structure, first observed in 1991 by Iijima \cite{ii}. Carbon nanotubes posses many unusual properties, which are valuable for nanotechnology, materials science and technology, electronics, and optics. They can be open-ended or closed-ended. Open-ended single-walled carbon  nanotubes are also called tubulenes. In this paper, we model benzenoid hydrocarbons and tubulenes by graphs.

Theoretical molecular structure-descriptors (also called topological indices) are graph invariants that play an important role in chemistry, pharmaceutical sciences, and in materials science and engineering since they can be used to predict physico-chemical properties of organic compounds. The most commonly studied molecular descriptor is the Wiener index introduced in 1947 \cite{Wiener}, which is defined as the sum of distances between all the pairs of vertices in a molecular graph. Wiener showed that the Wiener index is closely correlated with the boiling points of alkane molecules. Further work on quantitative structure-activity relationships showed that it is also correlated with other quantities, for example the parameters of its critical point, the density, surface tension, and viscosity of its liquid phase.

The Wiener polarity index of a graph is defined as the number of unordered pairs of vertices at distance three. In the best of our knowledge, Wiener already had some information about the applicability of this molecular descriptor. Later, Lukovits and Linert \cite{luko} demonstrated quantitative structure-property relationships for the Wiener polarity index in a series of acyclic and cycle-containing hydrocarbons. Also, Hosoya found a physico-chemical interpretation for this index, see \cite{hosoya}. Furthermore, the Wiener polarity index is closely related to the Hosoya polynomial \cite{hosoya0}, since it is exactly the coefficient before $x^3$ in this polynomial. In recent years, a lot of research has been done in investigating the Wiener polarity index of trees and unicyclic graphs, see \cite{ashrafi,deng,deng1,hou,lei,liu}. Also, the Nordhaus-Gaddum-type results for this index were considered in \cite{hua,zhang}. For some other recent investigations on the Wiener polarity index see  \cite{chen,hua1,ilic,ma,wang}. In \cite{beh} the Wiener polarity index was expressed by using the Zagreb indices and by applying this result, formulas for fullerenes and catacondensed benzenoid systems were found. In this paper, we choose a different approach and express the index by the number of hexagons and the Wiener polarity indices of smaller graphs.

We proceed as follows. In the following section, some basic definitions and notations are introduced. In Section 3, we develop a cut method for computing the Wiener polarity index of benzenoid systems and tubulenes. For a survey paper on the cut method see \cite{klavzar-2015} and some resent investigations on this topic can be found in \cite{cre-trat,cre-trat1,kla,tratnik2}. Then, the main result is used to find the closed formulas for the Wiener polarity index of benzenoid systems, zig-zag tubulenes, and armchair tubulenes. Finally, some ideas for the future work are presented.

\section{Preliminaries}

\noindent
A \textit{graph} $G$ is an ordered pair $G = (V, E)$ of a set $V$ of \textit{vertices} (also called nodes or points) together with a set $E$ of \textit{edges}, which are $2$-element subsets of $V$. For some basic concepts about graph theory see \cite{west}. Having a molecule, if we represent atoms by vertices and bonds by edges, we obtain a \textit{molecular graph}. The graphs considered in this paper are always simple and finite. The {\em distance} $d_G(x,y)$ between vertices $x$ and $y$ of a graph $G$ is the length of a shortest path between vertices $x$ and $y$ in $G$. If there is no confusion, we also write $d(x,y)$ for $d_G(x,y)$. Then the {\em Wiener polarity index} of a graph $G$, denoted by $W_p(G)$, is defined as 
$$W_p(G) = |\lbrace \lbrace u,v \rbrace \subseteq V(G) \,|\,d_G(u,v)=3 \rbrace|.$$

\noindent
Let ${\cal H}$ be the hexagonal (graphite) lattice and let $Z$ be a cycle on it. Then a {\em benzenoid system} $G$ is induced by the vertices and edges of ${\cal H}$, lying on $Z$ and in its interior, see Figure \ref{bezn_graf}. In such a case, every hexagon lying in the interior of $Z$ is called a \textit{hexagon of benzenoid system $G$}. The number of all the hexagons of $G$ will be denoted by $h(G)$ and the number of vertices in cycle $Z$ is denoted by $|Z|$. Moreover, a vertex of a benzenoid system $G$ is called \textit{internal} if it lies on exactly three hexagons of $G$.

 \begin{figure}[!htb]
	\centering
		\includegraphics[scale=0.6, trim=0cm 0cm 1cm 0cm]{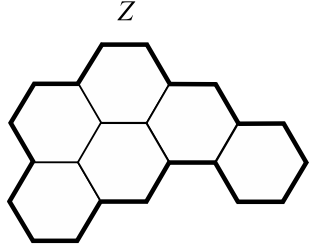}
\caption{Benzenoid system $G$ with the boundary cycle $Z$.}
	\label{bezn_graf}
\end{figure} 

\noindent
An \textit{elementary cut} of a benzenoid system $G$ is a line segment that starts at
the center of a peripheral (boundary) edge of a benzenoid system $G$,
goes orthogonal to it and ends at the first next peripheral
edge of $G$. In what follows, by an elementary cut we usually mean the set of all edges that intersect the elementary cut. Note that all benzenoid system are also partial cubes, which are defined as isometric subgraphs of hypercubes (a subgraph $H$ of $G$ is \textit{isometric} if for any two vertices $x,y \in V(H)$ it holds $d_H(x,y) = d_G(x,y)$) and represent a large class of graphs with a lot of applications. In particular, every elementary cut of a benzenoid system coincides with a $\Theta$-class. Recall that two edges $e_1 = u_1 v_1$ and $e_2 = u_2 v_2$ of a connected graph $G$ are in relation $\Theta$, $e_1 \Theta e_2$, if
$d_G(u_1,u_2) + d_G(v_1,v_2) \neq d_G(u_1,v_2) + d_G(u_1,v_2)$ and that in a partial cube this relation is always transitive. For more details about $\Theta$ relation and partial cubes see \cite{klavzar-book}. 
\bigskip

\noindent
Next, we define open-ended carbon nanotubes, also called tubulenes (see \cite{sa}), which represent an important family of chemical graphs. Choose any lattice point in the infinite and regular hexagonal lattice  as the origin $O$. Moreover, let $A$ be a point in the hexagonal lattice such that the graph distance between $O$ and $A$ is an even number greater or equal to six. In addition, let $\overrightarrow{a_1}$ and $\overrightarrow{a_2}$ be the two basic lattice vectors (see Figure \ref{nano}).
Obviously, there are integers $n,m$ such that $ \overrightarrow{OA} =n\overrightarrow{a_1}+m \overrightarrow{a_2}$. Draw two straight lines $L_1$ and $L_2$ passing through
$O$ and $A$ perpendicular to $O A$, respectively. By rolling up the hexagonal strip between $L_1$ and $L_2$ and gluing $L_1$ and $L_2$ such
that $A$ and $O$ superimpose, we can obtain a \textit{hexagonal tessellation} $\mathcal{HT}$ of the cylinder. $L_1$ and $L_2$ indicate the direction of
the axis of the cylinder. Using the terminology of graph theory, a {\em tubulene} $G$ is defined to be the finite graph induced by all
the hexagons of $\mathcal{HT}$ that lie between $c_1$ and $c_2$, where $c_1$ and $c_2$ are two vertex-disjoint cycles of $\mathcal{HT}$ encircling the axis of
the cylinder. Any such hexagon (between $c_1$ and $c_2$) is called a \textit{hexagon of a tubulene $G$} and the number of these hexagons will be denoted by $h(G)$.  The vector $\overrightarrow{OA}$ is called a {\em chiral vector} of $G$ and  the cycles $c_1$ and $c_2$ are the two open-ends of $G$. For any  tubulene $G$, if its chiral vector is $ n \overrightarrow{a_1} + m \overrightarrow{a_2}$, $G$ is called an $(n,m)$-type tubulene, see Figure \ref{nano}. 

\begin{figure}[!h]
	\centering
		\includegraphics[scale=0.7, trim=0cm 0cm 1cm 0cm]{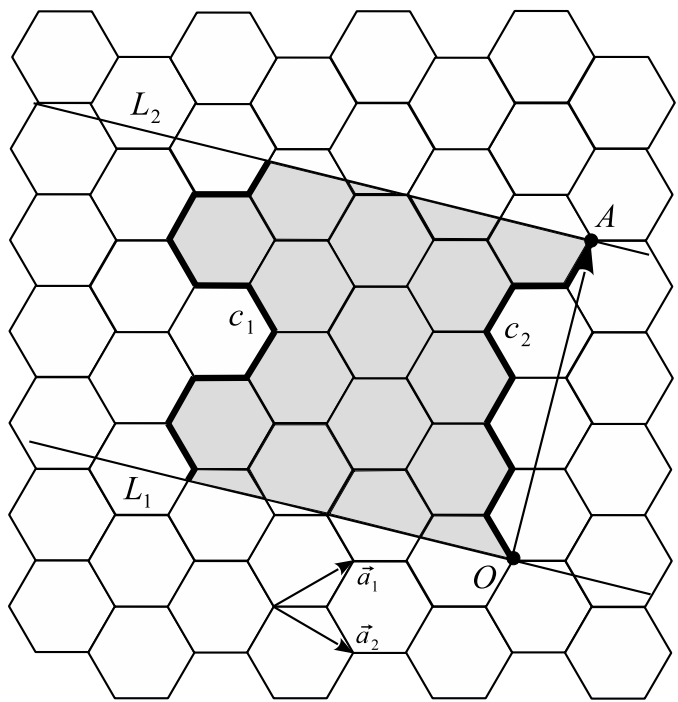}
\caption{A tubulene of type $(4,-3)$ with two basic lattice vectors.}
	\label{nano}
\end{figure}

\begin{remark} \label{opomba}
Sometimes, the definition of a tubulene does not contain the requirement that the graph distance between $O$ and $A$ is at least six (but very often, some other condition on $n$ and $m$ is added). However, if this distance is four (or even less), it can happen that two distinct hexagons have two common edges, which is not usual for the considered molecules. Moreover, by including such a requirement, a tubulene can not contain a cycle of length four (or less), which is essentially used in Theorem \ref{glavni}.
\end{remark}

\section{A cut method for benzenoid systems and tubulenes}

To develop a cut method for the Wiener polarity index, some definitions and lemmas are first needed. Obviously, in the regular hexagonal lattice there are exactly three directions of edges. Let $E_1'$, $E_2'$, and $E_3'$ be the sets of edges of the same direction. Moreover, if $G$ is a benzenoid system or a tubulene (drawn in the hexagonal lattice), we define $E_i(G) = E_i' \cap E(G)$ and often use $E_i$ instead of $E_i(G)$, where $i \in \lbrace 1,2,3 \rbrace$. One can notice that for a benzenoid system the set $E_i$ is the set of all elementary cuts in a given direction. For our consideration it is important to observe that for any benzenoid system or a tubulene $G$ it holds:
\begin{itemize}
\item [$(i)$] any two hexagons of $G$ have at most one edge in common (see Remark \ref{opomba}),
\item [$(ii)$] on any hexagon of $G$ there are exactly two edges from $E_1$, two edges from $E_2$, and two edges from $E_3$,
\item [$(iii)$] if $e,e' \in E_i$, $i \in \lbrace 1,2,3 \rbrace$, are two distinct edges, then they have no vertex in common.
\end{itemize}

\noindent
In the rest of the paper we will denote by $G-E_i$, $i \in \lbrace 1,2,3 \rbrace$, the graph obtained from $G$ by deleting all the edges from $E_i$. Also, for  $i \in \lbrace 1,2,3 \rbrace$ let ${\mathcal{C}}_i(G)$ (or simply ${\mathcal{C}}_i$) be the set of all connected components of the graph $G-E_i$. Furthermore, let ${\mathcal{C}} = {\mathcal{C}}_1 \cup {\mathcal{C}}_2 \cup {\mathcal{C}}_3$.

\begin{lemma} \label{lema00}
If $G$ is a benzenoid system, then every connected component of $G-E_i$, $i \in \lbrace 1,2,3 \rbrace$, is a path on at least three vertices.
\end{lemma}

\begin{proof}
It is obvious that every connected component is a path. Since every elementary cut goes through two opposite edges of a hexagon, every component must have at least three vertices. \qed
\end{proof}
\smallskip

\noindent
However, in the case of tubulenes a connected component can be a cycle (see, for example, zig-zag tubulenes in Section 5).

\begin{lemma}
If $G$ is a tubulene, then every connected component of $G-E_i$, $i \in \lbrace 1,2,3 \rbrace$, is a path or a cycle.
\end{lemma}

\begin{proof}
All vertices of degree three in $G$ have exactly one incident edge in $E_j$ for every $j \in \lbrace 1,2,3 \rbrace$. Therefore, after deleting all the edges of $E_i$ from $G$, every vertex of $G-E_i$ has degree at most two. Therefore, every connected component of $G-E_i$ is a path or a cycle. \qed
\end{proof}
\smallskip

A hexagon $h$ of the hexagonal lattice $\mathcal{H}$ is called \textit{external} for a benzenoid system $G$ if $h$ is not a hexagon of $G$ but $E(h) \cap E(G) \neq \emptyset$. Similarly, a hexagon $h$ of the hexagonal tessellation $\mathcal{HT}$ is called \textit{external} for a tubulene $G$ if $h$ is not a hexagon of $G$ but $E(h) \cap E(G) \neq \emptyset$. 
If $h$ is an external hexagon of a benzenoid system or a tubulene $G$, we notice that a largest connected component of the intersection $h \cap G$ is always a path. Especially important are external hexagons for which such a component is a path on at least four vertices.  Therefore, the set of all external hexagons $h$ for which the largest connected component of $h \cap G$ is a path $P_4$ will be denoted as $H_1(G)$. Also, the set $H_2(G)$ (or $H_3(G)$) is defined as the set of all external hexagons of $G$ for which the largest connected component of $h \cap G$ is a path $P_5$ (or $P_6$).  
Moreover, we denote the cardinality of $H_k(G)$ by $h_k(G)$, i.e.\,$h_k(G)=|H_k(G)|$ for $k \in \lbrace 1,2,3 \rbrace$.

In Figure \ref{bezn_graf1} we can see a benzenoid system $G$, where hexagons of $G$ are coloured grey and external hexagons of $G$ are white. Obviously, $h \in H_1(G)$, $h' \in H_2(G)$, and $h'' \in H_3(G)$.

 \begin{figure}[!htb]
	\centering
		\includegraphics[scale=0.6, trim=0cm 0cm 1cm 0cm]{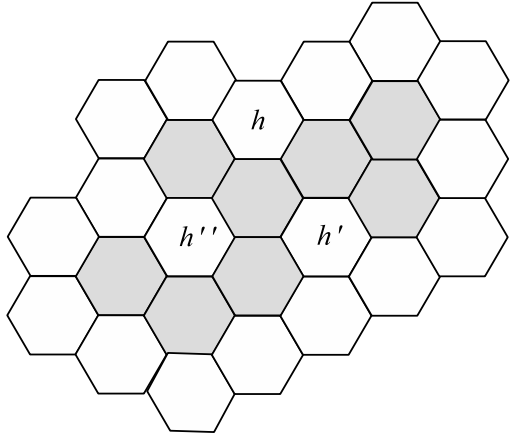}
\caption{Benzenoid system $G$ with external hexagons.}
	\label{bezn_graf1}
\end{figure} 

In order to develop a cut method for the Wiener polarity index, we next study the structure of shortest paths of length three.

\begin{lemma} \label{lema1}
Let $G$ be a benzenoid system or a tubulene and let $P$ be a shortest path of length three in $G$ such that $E(P) \cap E_i \neq \emptyset$ for any $i \in \lbrace 1,2,3 \rbrace$. Then all vertices and edges of $P$ lie on a hexagon of $G$ or on an external hexagon of $G$.
\end{lemma}

\begin{proof}
Denote the edges of $P$ by $e_1$, $e_2$, and $e_3$ such that $e_1$ is incident with $e_2$ and $e_2$ is incident with $e_3$. Without lost of generality assume that $e_i \in E_i$ for any $i \in \lbrace 1,2,3 \rbrace$. Obviously, there is exactly one hexagon $h$, which is a hexagon of $G$ or an external hexagon of $G$, such that $e_1$ and $e_2$ belong to $h$. Let $v$ be an end-point of $e_2$ such that $v$ is not an end-point of $e_1$ and let $e$ be an edge of $h$ such that $e$ is incident with $e_2$ and $e \neq e_1$. Since edges of any hexagon alternatively belong to $E_1$, $E_2$, and $E_3$, it follows that $e \in E_3$. Therefore, since no two edges from $E_3$ are incident, it follows that $e=e_3$ and the proof is complete. \qed
\end{proof}

\begin{lemma} \label{lema2}
Let $G$ be a benzenoid system or a tubulene and let $P$ be a shortest path of length three in $G$. Then exactly one of the following two statements holds:
\begin{itemize}
\item [$(i)$] $P$ belongs to exactly one hexagon, which is a hexagon of $G$ or an external hexagon of $G$,
\item [$(ii)$] there exists exactly one $i \in \lbrace 1,2,3 \rbrace$ such that $P$ belongs to exactly one connected component of $G - E_i$. 
\end{itemize}
\end{lemma}

\begin{proof}
If two edges $e$ and $e'$ are incident, they can not both belong to the set $E_i$, $i \in \lbrace 1,2,3 \rbrace$. Therefore, the edges of $P$ can not all belong to one set $E_i$. Hence, we have two possibilities:
\begin{itemize}
\item If $E(P) \cap E_i \neq \emptyset$ for any $i \in \lbrace 1,2,3 \rbrace$, then by Lemma \ref{lema1} $P$ belongs to (exactly) one hexagon of $G$ or to one external hexagon of $G$.
\item If there exists $i \in \lbrace 1,2,3 \rbrace$ such that $E(P) \cap E_i = \emptyset$ (note that such $i$ is unique), then all vertices and edges of $P$ belong to the graph $G-E_i$. Since $P$ is connected, it belongs to exactly one component of $G-E_i$.
\end{itemize}
Therefore, the proof is complete. \qed
\end{proof}

Finally, we are able to prove the main result of this section, which enables us to compute the Wiener polarity index of a benzenoid system or a tubulene.

\begin{theorem} \label{glavni}
Let $G$ be a benzenoid system or a tubulene. Then
$$W_p(G) = 3h(G) + h_1(G) + 2h_2(G) + 3h_3(G) + \sum_{C \in {\mathcal{C}}}W_p(C).$$
\end{theorem}

\begin{proof}
Denote by $V'$ the set of all unordered pairs of vertices $\lbrace u,v \rbrace$ such that $u,v$ lie on the same hexagon of $G$ and $d(u,v)=3$. Similarly, denote by $V''$ the set of all unordered pairs of vertices $\lbrace u,v \rbrace$ such that $u,v$ lie on the same external hexagon of $G$ and $d(u,v)=3$. Moreover, let $V'''$ be the set of all unordered pairs of vertices $\lbrace u,v \rbrace$ such that $d(u,v)=3$ and $u,v$ do not lie on the same hexagon. It follows
$$W_p(G) = |V'| + |V''| + |V'''|.$$

Obviously, for any hexagon of $G$, there are three unordered pairs of vertices on $h$ at distance three. On the other hand, if for two vertices $u,v$ on the same hexagon it holds $d(u,v)=3$, then by Lemma \ref{lema2} this hexagon is unique. Therefore, $|V'|=3h(G)$.

If $h \in H_1(G)$, there is exactly one pair of vertices $u,v \in V(h) \cap V(G)$ at distance three. If $h \in H_2(G)$ we have two such pairs and for $h \in H_3(G)$ there are exactly three such pairs. Hence, $|V''| = h_1(G) + 2h_2(G) + 3h_3(G).$

If $u,v$ are vertices at distance three that do not lie on the same hexagon, then by Lemma \ref{lema2} they belong to exactly one connected component of $G-E_i$, where $i \in \lbrace 1,2,3 \rbrace$. For the contrary, let $u,v$ belong to exactly one connected component $C$ of $G-E_i$ for some $i \in \lbrace 1,2,3 \rbrace$ and $d_C(u,v)=3$. We will show that $d_G(u,v)=3$. Consider two cases:
\begin{itemize}
\item If $G$ is a benzenoid system, it obviously follows $d_G(u,v)=3$ since there is only one shortest path between $u$ and $v$ in $G$.
\item If $G$ is a tubulene, it could happen that $d_G(u,v)=1$ (the distance must be odd since every tubulene is a bipartite graph, see Theorem 3.1 in \cite{tra-zi}). In such a case, $u$ and $v$ are adjacent with an edge and we obtain a cycle of length four in $G$, which is a contradiction by Remark \ref{opomba}. Therefore, $d_G(u,v)=3$.
\end{itemize}

Since $d_G(u,v)=3$ and $u,v$ belong to one connected component of $G-E_i$, they can not both belong to one hexagon. Hence, $|V'''| = \sum_{C \in {\mathcal{C}}}W_p(C)$ and we are finished. \qed

\end{proof}
\smallskip

The previous result enables us to reduce the problem of computing the Wiener polarity index of $G$ to the problem of computing the Wiener polarity indices of paths or cycles, which is trivial. In particular, if $P_n$ is a path on $n \geq 1$ vertices, then

\begin{equation} \label{pot}
W_p(P_n) = \left\{ \begin{array}{ll} 0, & n=1,2 \\ 
 n-3, & n \geq 3  \end{array} \right.
\end{equation}

\noindent
and if $C_n$ is a cycles on $n \geq 3$ vertices, we have
\begin{equation} \label{cik}
W_p(C_n) = \left\{ \begin{array}{ll} 0, & n=3,4,5 \\ 
 3, & n = 6 \\
 n, & n \geq 7.  \end{array} \right.
\end{equation}

\section{A closed formula for benzenoid systems}

In this section we focus on benzenoid system and use the developed cut method to obtain a closed formula for the Wiener polarity index of an arbitrary benzenoid system. Throughout the section we denote by $\alpha_i(G)$, $i \in \lbrace 1,2,3 \rbrace$, the number of elementary cuts of a benzenoid system $G$ that contain only edges from $E_i$ (the elementary cuts in a given direction). First, three auxiliary results are needed.

\begin{lemma} \cite{gut} \label{lema31}
Let $G$ be a benzenoid system with $n_i(G)$ internal vertices. Then $n(G) = 4h(G) + 2 - n_i(G)$.
\end{lemma}

\begin{lemma} \label{lema3}
Let $G$ be a benzenoid system with the boundary cycle $Z$. Then
$$|Z| = 2(\alpha_1(G) + \alpha_2(G) + \alpha_3(G)).$$
\end{lemma}

\begin{proof}
Obviously, every elementary cut of $G$ intersects the boundary cycle $Z$ exactly twice. Since the number of elementary cuts is $\alpha_1(G) + \alpha_2(G) + \alpha_3(G)$, we have that the number of edges in $Z$ is
$2(\alpha_1(G) + \alpha_2(G) + \alpha_3(G))$. Hence, the result follows. \qed
\end{proof}

\begin{lemma} \label{lema4}
Let $G$ be a benzenoid system. Then the graph $G-E_i$, $i \in \lbrace 1,2,3 \rbrace$, has exactly $\alpha_i(G)+1$ connected components.
\end{lemma}

\begin{proof}
When all the edges from one elementary cut are deleted, the obtained graph has exactly two connected components. Therefore, if we delete elementary cuts one by one, at every step we get one additional component. After deleting all the elementary cuts we have $\alpha_i(G) +1$ connected components. \qed
\end{proof}
\smallskip

\noindent
Note that Lemma \ref{lema4} also follows from the fact that the \textit{quotient graph} $T_i$, which has connected components of $G-E_i$ as vertices, two such components being adjacent whenever there is an edge in $E_i$ connecting them, is always a tree \cite{chepoi-1996}. In such a tree, every vertex represents a connected component of $G-E_i$ and every edge represents an elementary cut in direction $E_i$. For the details see \cite{chepoi-1996}.
\smallskip

\noindent
Finally, we are able to show the main result of this section.

\begin{theorem} \label{iz_ben}
Let $G$ be a benzenoid system. Then 
$$W_p(G) = 9h(G) + h_1(G) + 2h_2(G) + 3h_3(G) - 6.$$

\end{theorem}

\begin{proof}
Let $i \in \lbrace 1,2,3 \rbrace$. Since by Lemma \ref{lema4} the graph $G-E_i$ has exactly $\alpha_i(G)+1$ connected components, we denote by $n^i_1, \ldots, n^i_{\alpha_i(G)+1}$ the number of vertices in the connected components of $G-E_i$. Obviously
\begin{equation} \label{en1}
\sum_{j=1}^{\alpha_i(G)+1}n^i_{j}= n(G).
\end{equation}
By Theorem \ref{glavni},
\begin{equation} \label{en2}
W_p(G) = 3h(G) + h_1(G) + 2h_2(G) + 3h_3(G) + \sum_{P \in {\mathcal{C}_1}}W_p(P) + \sum_{P \in {\mathcal{C}_2}}W_p(P) + \sum_{P \in {\mathcal{C}_3}}W_p(P).
\end{equation}

\noindent
Since any $P \in {\mathcal{C}_i}$, $i \in \lbrace 1,2,3 \rbrace$, has at least three vertices (see Lemma \ref{lema00}), by Equation \ref{pot} and Equation \ref{en1} we get
$$\sum_{P \in {\mathcal{C}_i}}W_p(P) = \sum_{j=1}^{\alpha_i(G)+1}(n^i_j - 3) = n(G) - 3(\alpha_i(G)+1).$$

\noindent
By inserting this into Equation \ref{en2} we deduce
$$W_p(G) = 3h(G) +h_1(G) + 2h_2(G)+3h_3(G) + 3n(G) - 3(\alpha_1(G) + \alpha_2(G) + \alpha_3(G)) - 9$$

\noindent
and then, using Lemma \ref{lema3} one can obtain
$$W_p(G) = 3h(G) +h_1(G) + 2h_2(G) + 3h_3(G) + 3n(G) - \frac{3|Z|}{2} - 9.$$

\noindent
Obviously, the number of internal vertices of $G$ is exactly $n(G) - |Z|$. Hence, Lemma \ref{lema31} implies $|Z|=2n(G)-4h(G)-2$. Inserting this in the previous equation we finally get
$$W_p(G) = 9h(G) + h_1(G) + 2h_2(G) + 3h_3(G) - 6,$$
which finishes the proof. \qed
\end{proof}
\smallskip

Note the the number of hexagons of a benzenoid system can be obtained from its boundary edges code (see \cite{kovic} for the details). Therefore, by Theorem \ref{iz_ben} the Wiener polarity index of a benzenoid system can be determined by the shape of its boundary. To  show how Theorem \ref{iz_ben} can be used, we demonstrate it on one simple example. Let $G$ be a benzenoid system from Figure \ref{bezn_graf1}. We can immediately see that $h(G) = 8$ and $h_1(G)=h_2(G)=h_3(G)=1$. Therefore, by Theorem \ref{iz_ben} we obtain
$$W_p(G) = 9 \cdot 8 + 1 + 2 \cdot 1 + 3\cdot 1 - 6 = 72.$$

\section{The Wiener polarity index of zig-zag and armchair tubulenes}

The aim of this section is to use the obtained cut method to calculate closed formulas for the Wiener polarity index of zig-zag and armchair tubulenes.

\subsection{Zig-zag tubulenes}

If $G$ is a $(n,m)$-type tubulene where $n=0$ or $m=0$, we call it a \textit{zig-zag tubulene}. Let $G$ be a zig-zag tubulene such that $c_1, c_2$ are the shortest possible cycles encircling the axis of the cylinder (see Figure \ref{zig-zag}). If $G$ has $r$ layers of hexagons, each containing exactly $h$ hexagons, then we denote it by $ZT(r,h)$, i.e.\, $G=ZT(r,h)$. We always assume $r \geq 1$ and $h \geq 3$ (see Remark \ref{opomba}).

\begin{figure}[!htb]
	\centering
		\includegraphics[scale=0.6, trim=0cm 0cm 1cm 0cm]{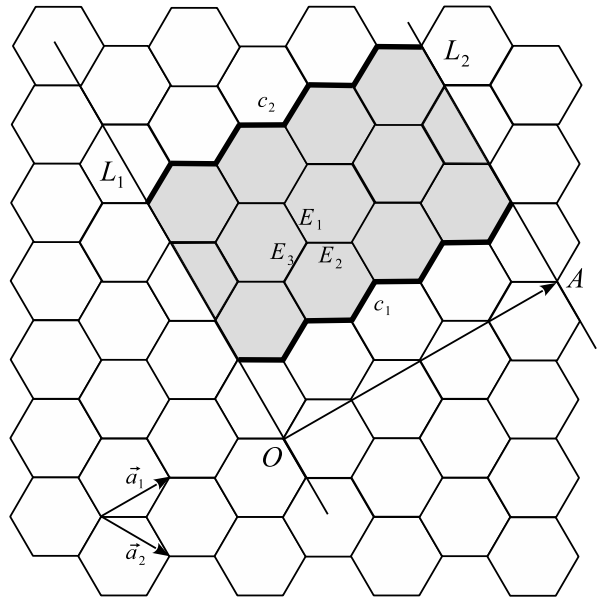}
\caption{Zig-zag tubulene $ZT(3,4)$ with three directions of edges ($E_1$, $E_2$, and $E_3$).}
	\label{zig-zag}
\end{figure}

Obviously, the graph $G-E_1$ has $r+1$ connected components, each isomorphic to a cycle on $2h$ vertices. Moreover, the graph $G-E_2$ has $h$ connected components, each isomorphic to a path on $2r+2$ vertices, and the same holds for the graph $G-E_3$. Also, we notice that $h_1(G)=h_2(G)=h_3(G)=0$. Therefore, by Theorem \ref{glavni} we obtain

$$W_p(ZT(r,h)) = 3rh + (r+1)W_p(C_{2h}) + 2hW_p(P_{2r+2}).$$ 

\noindent
Using Equation \ref{pot} and Equation \ref{cik} we conclude

\begin{equation*} 
W_p(ZT(r,h)) = \left\{ \begin{array}{ll} 24r-3, & h=3 \\ 
 9rh, & h \geq 4.  \end{array} \right.
\end{equation*}

\subsection{Armchair tubulenes}

If $G$ is a $(n,m)$-type tubulene where $n=m$, we call it an \textit{armchair tubulene}. Let $G$ be an armchair tubulene such that $c_1$ and $c_2$ are the shortest possible cycles encircling the axis of the cylinder and such that there is the same number of hexagons in every column of hexagons (see Figure \ref{armchair}). If $G$ has $r$ vertical columns of hexagons, each containing exactly $h$ hexagons, then we denote it by $AT(r,h)$, i.e.\,$G=AT(r,h)$. Obviously, $r$ must be an even number. Note that $AT(r,h)$ is a $(\frac{r}{2},\frac{r}{2})$-type tubulene. Furthermore, we always assume $h \geq 1$ and $r \geq 4$ (see Remark \ref{opomba}).

 \begin{figure}[!htb]
	\centering
		\includegraphics[scale=0.6, trim=0cm 0cm 1cm 0cm]{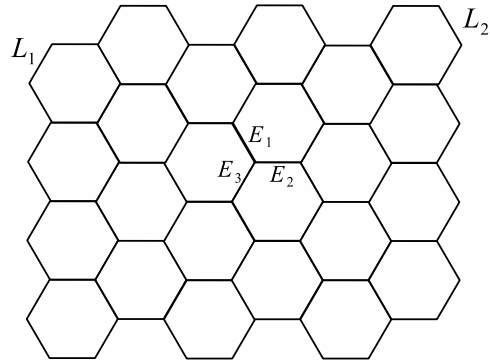}
\caption{Armchair tubulene $AT(6,4)$ with three directions of edges ($E_1$, $E_2$, and $E_3$). Curves $L_1$ and $L_2$ are joined together.}
	\label{armchair}
\end{figure}

First, we observe $h_1(G) = r$ and $h_2(G)=h_3(G)=0$. Also, it is not difficult to see that the graph $AT(r,h) - E_1$ has exactly $\frac{r}{2}$ connected components, all isomorphic to a path on at least three vertices, see Figure \ref{armchair_com}.

 \begin{figure}[!htb]
	\centering
		\includegraphics[scale=0.6, trim=0cm 0cm 1cm 0cm]{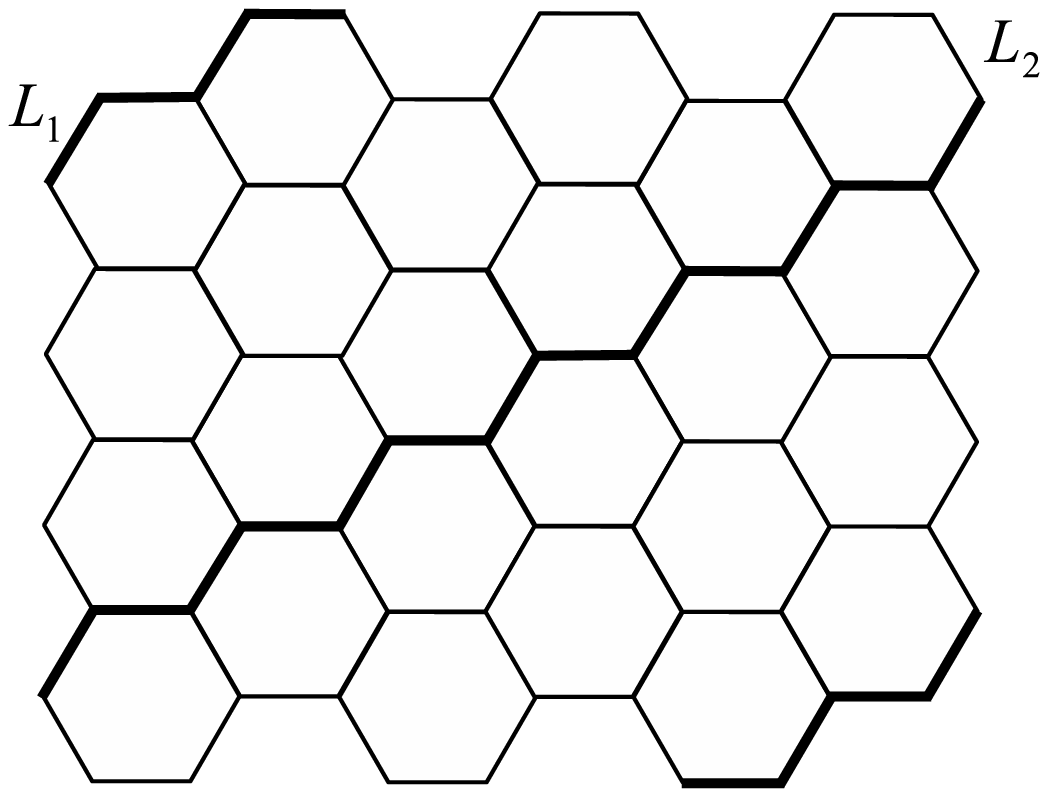}
\caption{The bold edges in $G=AT(6,4)$ represent a connected component of the graph $G-E_1$.}
	\label{armchair_com}
\end{figure} 

\noindent
Denote the number of vertices in these components by $n_1, \ldots, n_{\frac{r}{2}}$. Therefore, by Equation \ref{pot}
$$\sum_{P \in \mathcal{C}_1}W_p(P) = (n_1 + \cdots + n_{\frac{r}{2}}) - \frac{3r}{2} = |V(AT(r,h))| - \frac{3r}{2} = r(2h+2) - \frac{3r}{2}$$
and the same result holds for $AT(r,h) - E_3$. On the other hand, the graph $G- E_2$ has exactly $r$ connected components, each isomorphic to a path on $2h+2$ vertices. Therefore, by Equation \ref{pot}

$$\sum_{P \in \mathcal{C}_2}W_p(P) = r(2h-1).$$

\noindent
Finally, using Theorem \ref{glavni} we conclude
$$W_p(AT(r,h)) = 9rh+r.$$

\section{Concluding remarks}

In the paper we have developed a method for computing the Wiener polarity index of benzenoid system and tubulenes. Naturally, one can ask to which other planar or molecular graphs a similar method can be applied. Furthermore, in the case of benzenoid systems the elementary cuts play an important role. Therefore, since all benzenoid systems are partial cubes (and in partial cubes elementary cuts represent $\Theta$-classes), it would be interesting to find a generalization of the described method to all partial cubes. 

\section*{Acknowledgment} 

\noindent The author was financially supported by the Slovenian Research Agency.

\end{document}